\documentclass[letterpaper, 10 pt, conference]{ieeeconf}
\IEEEoverridecommandlockouts                              
\usepackage{xifthen}
\title{\LARGE \bf
Guaranteed Robust Nonlinear MPC via Disturbance Feedback
}
\newboolean{anonymize}
\setboolean{anonymize}{false}
\ifthenelse{\boolean{anonymize}}{
  \author{
  }%
}{
\author{Antoine P. Leeman$^{1}$ \and Johannes Köhler$^{1}$ \and Melanie N. Zeilinger$^{1}$
\thanks{This work has been supported by the European Space Agency under OSIP 4000133352, the Swiss Space Center, and the Swiss National Science Foundation under NCCR Automation (grant agreement 51NF40 180545).}
\thanks{$^1$ Institute for Dynamic Systems and Control,
ETH Zürich, Zürich 8053, Switzerland. Email: \texttt{\{aleeman, jkoehle, mzeilinger\}@ethz.ch}}
}
}
\usepackage{array}
\usepackage{cite}
\usepackage{amssymb,amsfonts}
\usepackage{textcomp}
\usepackage{xcolor}
\usepackage{graphics}
\usepackage{bm}
\usepackage{graphicx}
\usepackage[intlimits]{amsmath}
\usepackage{mathtools}
\usepackage{comment}
\usepackage{printlen}
\usepackage{placeins}
\usepackage[capitalise]{acro}
\usepackage{ifthen}
  \usepackage{soul}
\usepackage{booktabs}
\usepackage{balance}
\usepackage{multirow}
\usepackage{multicol}

\usepackage{booktabs}
\usepackage{siunitx}
\usepackage{tikz, adjustbox}
\usetikzlibrary{positioning,fit,backgrounds,arrows.meta,calc}
\usepackage[most]{tcolorbox}
\usepackage{xcolor}
\usepackage{wrapfig}

\usepackage{makecell}
\usepackage{hyperref}
\DeclareAcronym{RMPC}{
  short = RMPC,
  long  = robust model predictive control,
}
\DeclareAcronym{SLS}{
  short = SLS,
  long  = system level synthesis,
}
\DeclareAcronym{MPC}{
  short = MPC,
  long  = model predictive control,
}

\DeclareAcronym{LTV}{
  short = LTV,
  long  = linear time-varying,
}

\newcommand{\Z}{\mathbf{z}}
\newcommand{\V}{\mathbf{v}}

\renewcommand{\P}{ \Phi}
\newcommand{\Px}{\P^\x}
\newcommand{\Pu}{\P^\u}

\newcommand{\nx}{{n_\mathrm{x}}}
\renewcommand{\nu}{{n_\mathrm{u}}}

\newcommand{\nw}{{n_\mathrm{x}}}

\newcommand{\R}{\mathbb{R}}

\newcommand{\sv}[2]{\left(#1,#2\right)}
\newcommand{\T}{^\top}

\newcommand{\Xf}{\mathcal{X}_\f}

\newcommand{\x}{\mathrm{x}}
\renewcommand{\u}{\mathrm{u}}

\newcommand{\f}{\mathrm{f}}

\newcommand{\tube}{\bm \tau}

\newcommand{\B}{ \mathcal{B}}
\newcommand{\defmath}{\vcentcolon =}

\newtheorem{remark}{Remark}

\newtheorem{proposition}{Proposition}
\newtheorem{theorem}{Theorem}
\newtheorem{assumption}{Assumption}

\DeclareAcronym{NLP}{
  short = NLP,
  long  = nonlinear programming,
}

\DeclareAcronym{NMPC}{
  short = NMPC,
  long  = nonlinear model predictive control,
}
\DeclareAcronym{RNMPC}{
  short = RNMPC,
  long  = robust nonlinear model predictive control,
}
\DeclareAcronym{ISS}{
  short = ISS,
  long  = input-to-state stability,
}
\DeclareAcronym{RPI}{
  short = RPI,
  long  = robust positively invariant,
}
\DeclareAcronym{QP}{
  short = QP,
  long  = quadratic program,
}
\DeclareAcronym{SOCP}{
  short = SOCP,
  long  = second-order cone program,
}
\DeclareAcronym{SCP}{
  short = SCP,
  long  = sequential convex programming,
}
\DeclareAcronym{LQR}{
  short = LQR,
  long  = linear quadratic regulator,
}
\DeclareAcronym{iLQR}{
  short = iLQR,
  long  = iterative linear quadratic regulator,
}
\DeclareAcronym{ESA}{
  short = ESA,
  long  = European Space Agency,
}
\DeclareAcronym{GPU}{
  short = GPU,
  long  = graphics processing unit,
}
\DeclareAcronym{JIT}{
  short = JIT,
  long  = just-in-time,
}
\DeclareAcronym{LTI}{
  short = LTI,
  long  = linear time-invariant,
}
\DeclareAcronym{RCI}{
  short = RCI,
  long  = robust control invariant,
}
\DeclareAcronym{RL}{
  short = RL,
  long  = reinforcement learning,
}

\DeclareAcronym{RTI}{
  short = RTI,
  long  = real-time iteration,
}

\DeclareAcronym{SQP}{
  short = SQP,
  long  = sequential quadratic programming,
}

\begin{document}

\maketitle
\thispagestyle{empty}
\pagestyle{empty}

\begin{abstract}
Robots must satisfy safety-critical state and input constraints despite disturbances and model mismatch.
We introduce a \ac{RMPC} formulation that is fast, scalable, and compatible with real-time implementation.
Our formulation guarantees robust constraint satisfaction, \ac{ISS} and recursive feasibility.
The key idea is to decompose the uncertain nonlinear system into (i) a nominal nonlinear dynamic model, (ii) disturbance-feedback controllers, and (iii) bounds on the model error. These components are optimized jointly using sequential convex programming. 
The resulting convex subproblems are solved efficiently using a recent disturbance-feedback {MPC} solver.
The approach is validated across multiple dynamics, including a rocket-landing problem with steerable thrust. An open-source implementation is available at  \ifthenelse{\boolean{anonymize}}{\url{https://anonymous.4open.science/r/robust-nonlinear-mpc/}.}{\url{https://github.com/antoineleeman/robust-nonlinear-mpc}.}
\end{abstract}
\acresetall 

\section{Introduction}
Autonomous robots, whether agile drones, wheeled machines, or (autonomous) spacecrafts, must operate in dynamic and uncertain environments while satisfying strict safety and performance requirements~\cite{brunke2022safe}.
In addition, model mismatch arises naturally due to many factors, such as wind gusts, actuators misalignments, or unmodelled frictions.
In robotics applications, disturbances such as wind gusts, actuator misalignments, or unmodeled friction are typically handled by introducing ad hoc safety margins in the control design, resulting in slower motions, reduced maneuverability, and under-utilization of the system’s capabilities.

\Ac{RL}, often with domain randomization, has recently shown success in achieving robust sim2real performance\cite{jenelten2024dtc}, particularly in contact-rich tasks. While learned policies can be executed in real time, training  requires extensive offline computation, careful reward design, and heuristics to ensure convergence.

\begin{figure}[ht!]
    \centering
    \includegraphics[width = \linewidth]{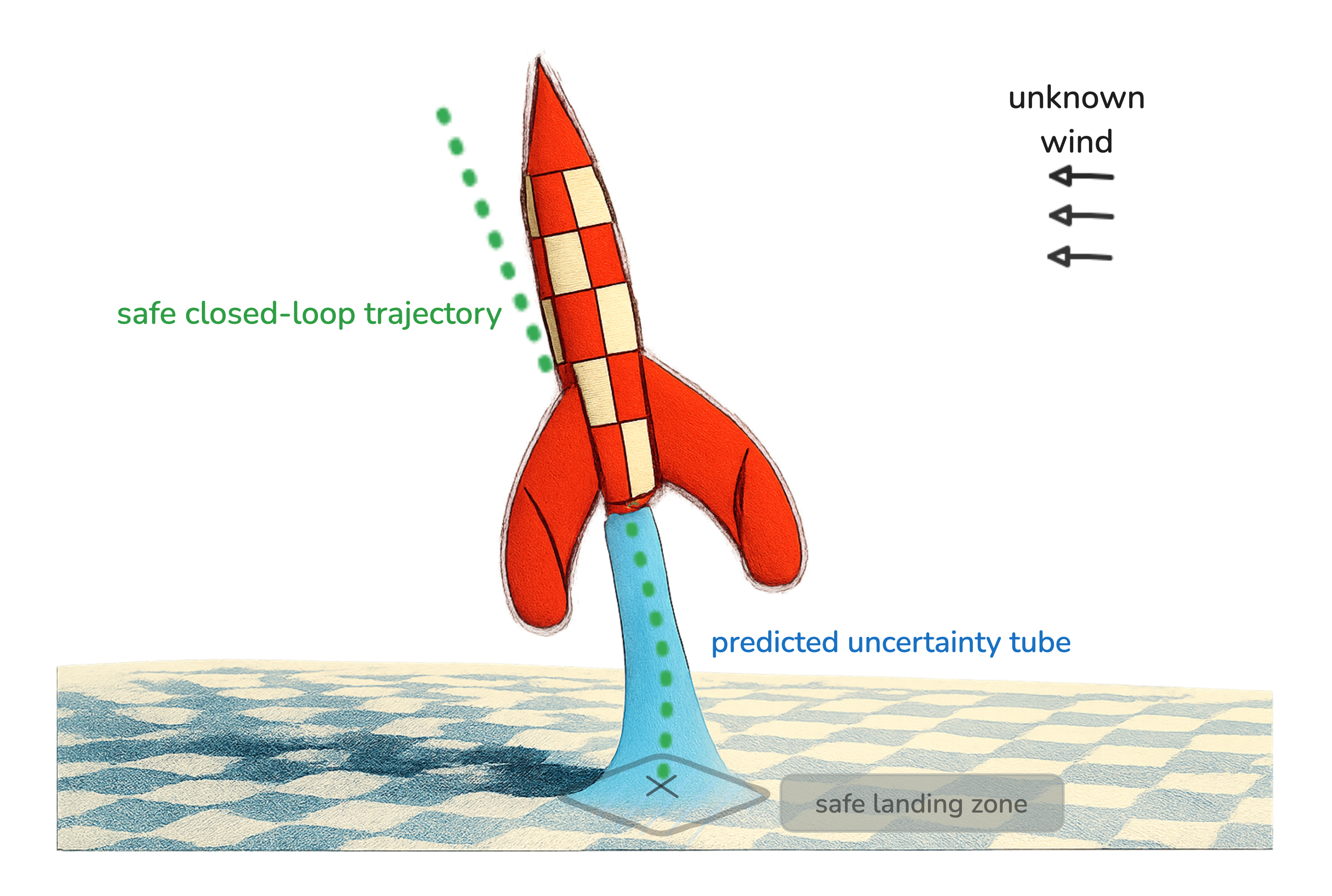}
    \caption{%
    Illustration of \ac{RMPC} for real-time rocket control with steerable thrust in a safe landing zone.
    The formulation jointly optimizes a nominal trajectory and  uncertainty tubes within the constraints to ensure safety. The figure is based on an edited MuJoCo rendering.}
    \label{fig:tintin}
\end{figure}
In contrast, trajectory-optimization methods enforce constraint satisfaction for nonlinear dynamics and are widely used in practice as a model-based control technique based on \ac{SCP}~\cite{malyuta2022convex}.
However, these methods typically do not ensure safety or stability in the presence of disturbances, which is critical for real-world deployment.

In this paper, we introduce a scalable \ac{RMPC} formulation for nonlinear systems that is safe-by-design.
\Ac{RMPC} commonly accounts for disturbances by predicting a set containing all possible future states~\cite{rawlings2017model}. 
To reduce conservatism, these robust predictions are based on closed-loop predictions and a corresponding feedback law is typically optimized offline, e.g., using contraction metrics~\cite{singh2023robust, zhao2022tube}.  
However, fixing the feedback a priori can limit closed-loop performance and the offline computations also limit scalability. 
\Ac{RMPC} approaches that optimize feedback laws to reduce conservatism have been proposed in~\cite{Messerer2021AnFeedback, Kim2022JointNonlinearities}, which rely on (conservative) sequential over-approximations of the robust predictions. 
In contrast, the disturbance feedback MPC~\cite{GOULART2006523} framework  framework (also known as \ac{SLS}~\cite{anderson2019system}) provides an exact characterization of the robust prediction for \ac{LTV} systems, thereby avoiding this compounding effect. Recent extensions~\cite{leeman2025robust_TAC, leister2024robust} further enable its application to nonlinear systems.
While these formulations improve performance compared to fixed policy approaches, they generally do not provide guarantees of recursive feasibility or stability. However, such guarantees are crucial, since loss of feasibility at any step can cause safety constraint violations.
\textbf{Contribution}: Building on the nonlinear SLS formulation in~\cite{leeman2025robust_TAC}, 
we propose a fast and scalable \ac{RMPC} formulation for nonlinear systems with robust closed-loop guarantees. 
Our approach jointly optimizes the nominal nonlinear trajectory, a disturbance-feedback controller, and an upper bound on the prediction error.

\begin{itemize}
    \item Formal guarantees are provided, i.e., robust constraint satisfaction, recursive feasibility (Thm. \ref{thm:rec_feas}), and \ac{ISS} (Thm. \ref{thm:iss}). Recursive feasibility is ensured by a novel treatment of the mismatch with respect to the nonlinear nominal prediction. 
\item An efficient \ac{SCP} algorithm tailored to the robust MPC formulation is provided to enable real-time deployment. Each iteration consists of solving a nominal trajectory optimization with a \ac{QP}, updating a disturbance-feedback controller via Riccati recursions, and evaluating Jacobians of the nonlinear dynamics. The design is general and  the provided code can be directly applied to systems with large state and input dimensions and long prediction horizons.
\item Real-time feasibility (computation times) is demonstrated across different dynamics, including a quadcopter and a rocket landing. Robust performance is validated on the rocket-landing problem with steerable thrust including actuator dynamics, illustrated in Fig.~\ref{fig:tintin}, demonstrating robust constraint satisfaction with an average total latency of 19.7 \texttt{[ms]} per iteration.
A comparison to a soft-constrained MPC baseline highlights increased safety and stability of the proposed approach.
\end{itemize}

\textbf{Notation}:
For vectors or matrices $a$ and $b$ with the same number of rows, we denote their horizontal concatenation by $[a,~b]$. 
We denote stacked vectors or matrices by $\sv{a}{b} = [a\T,~b\T]\T$. For a vector $r\in \R^n$, we denote its $i^\text{th}$ component by $r_i$. 
For a sequence of matrices $M_{k,j}\in\R^{p\times q}$, indexed by $k > j\ge 0$, 
we define the shorthand horizontal concatenation $
M_{(k)} \defmath [M_{k,k-1},~M_{k,k-2},~\dots,~M_{k,0}] \in \R^{p\times kq}.$
For a vector $v\in\R^n$, we write its 1-norm as $\|v\|_1 = |v_1| + \ldots + |v_n|$ and its infinity norm as $\|v\|_\infty = \max_{i=1,\ldots,n} |v_i|$.
For a matrix $M\in\R^{m\times n}$, the matrix infinity norm is $\|M\|_\infty=\max_i \sum_j |M_{ij}|$.
We denote sets with calligraphic letters, e.g., $\mathcal{W}\subseteq \R^n$. 
Let $\B^{m}$ be the unit ball defined by $\B^{m} \defmath \{d\in\R^m|~ \|d\|_\infty \le 1\}$. The Minkowski sum of two sets $\mathcal{A}, \mathcal{D} \subseteq \R^n$ is defined as $\mathcal{A} \oplus \mathcal{D} \defmath \{a + d | a \in \mathcal{A}, d \in \mathcal{D} \}$.
For a vector-valued function $f: \mathbb{R}^n \rightarrow \mathbb{R}^q$, we denote the Jacobian by $\frac{\partial \phi}{\partial x}|_x \in \R^{q\times n}$. For the $i^\text{th}$ component $(f)_i$, we denote its Hessian by
$\frac{\partial^2 (f)_i}{\partial x^2}|_x \;\in\; \R^{n \times n}$.\section{Problem Setup}

We consider a discrete-time nonlinear system
\begin{equation}
    x(t+1)= f(x(t),u(t)) + E(x(t), u(t))\cdot w(t),~x(0)=\bar x,\label{eq:nonlinear_dyn}
\end{equation}
with state $x(t) \in \R^\nx$, input $u(t) \in \R^\nu$, invertible disturbance matrix $E(x,u) \in  \R^{\nx \times \nw}$, time $t\in\mathbb{N}_0$, and norm-bounded disturbance $w(t)\in\B^\nw$ with $\B^\nw$ the unit-norm ball.
We assume that the dynamics $f: \R ^{\nx} \times \R^{\nu} \rightarrow \R^\nx$ are two times continuously differentiable and the state $x(t)$ can be measured. We impose point-wise in time state and input constraints
\begin{equation}
    {(x(t),u(t)) \in \mathcal{C}},\quad \forall t\in\mathbb{N}_0\label{eq:prob_form_cons}
\end{equation}
that should be satisfied for any realization of the disturbance $w(t)\in \B^\nw$. The constraint set $\mathcal{C}$ is given by a compact polytope
\begin{equation}
\begin{aligned}
    \mathcal{C}& \defmath  \{(x,u)\in \R^{\nx + \nu}|~c_i\T(x,u) + b_i \le 0, ~ i = 1, \ldots, n_\textrm{c}\}.\label{eq:constraints}
    \end{aligned}
\end{equation}
with $c_i \in \R^{\nx + \nu}$, and $b_i \in \R$ and $0 \in \mathcal{C}$.
We consider the problem of stabilizing the origin with $f(0,0)=0$ and a quadratic cost with positive definite matrices $Q,R$:
\begin{equation}
\label{cost:setup}
\ell(x,u)=x^\top Q x+ u^\top R u.
\end{equation}

In the next section, we introduce the finite-horizon optimization problem solved at each time step, which optimizes affine disturbance-feedback controllers, so-called tube controllers, of the form
\begin{equation}
    u_k = \pi_k(x_0, \ldots, x_k), \quad k = 0,\ldots,T-1,
\end{equation}
and ensures robust constraint satisfaction.
This construction naturally yields two layers of feedback: (i) an outer \ac{MPC} feedback from re-solving the optimization at each sampling time, and (ii) an inner tube-controller for the disturbance propagation.

\section{Nonlinear robust predictions with disturbance feedback}
\label{sec:rmc}
In this section, the proposed robust MPC formulation is presented.
In Section~\ref{sec:sls}, we parameterize tube-controllers using disturbance feedback and leverage existing results \cite{leeman2023robust_CDC} in over-approximating (disturbance) reachable sets of nonlinear systems. In Section~\ref{sec:finite_horizon_robust_nonlinear_optimal_control}, we formulate the corresponding joint optimization problem.

\subsection{Robust reachable sets}
\label{sec:sls}
First, we derive a robust reachable set over the finite prediction horizon $T\in\mathbb{N}_0$. 
Our robust prediction leverages a nominal system
\begin{equation}
    z_{k+1} = f(z_k, v_k),\quad k =0, \ldots, T-1
    \label{eq:nom_nonlinear}
\end{equation}
where $z_k\in\R^\nx$ is the nominal state, and $v_k\in\R^\nu$ is the nominal input. 
The linearization around the nominal dynamics is defined as:\begin{equation}
\begin{gathered}
A(z, v) \;\defmath\;
  \left.\frac{\partial f}{\partial x}\right|_{\substack{x=z \\ u=v}}, \quad
B(z, v) \;\defmath\;
  \left.\frac{\partial f}{\partial u}\right|_{\substack{x=z \\ u=v}}. \\[6pt]
\end{gathered}
\label{eq:jac}
\end{equation}
We introduce an \ac{LTV} approximation of the nonlinear system~\eqref{eq:nonlinear_dyn}
\begin{equation}
\label{eq:LTV_aux}
    \widehat{\Delta x}^+= A(z,v)\Delta x+B(z,v)\Delta u,
\end{equation}
where $\widehat{\Delta x}$ is the approximate prediction and $\Delta x,\Delta u$ denote the deviations from the nominal trajectory. 

\begin{assumption}
\label{assum:sigma}
    The error between the true system~\eqref{eq:nonlinear_dyn} and its \ac{LTV} approximation \eqref{eq:LTV_aux} can be characterized by a function $\sigma: \R^{\nx} \times \R^\nu \times \R \rightarrow \R^\nx$ satisfying 
\begin{equation}
\begin{aligned}
&\bigl| f(z,v)+\widehat{\Delta x}^+-f(x,u) \bigr| \;\le\; \sigma(z,v,\tau)
\end{aligned}
\label{eq:LTV_uncertain}
\end{equation}
for all $(x,u), (z,v) \in \mathcal{C}$, $w\in \B^{\nw}$ with
 $\tau\in \R$ satisfying
\begin{equation}
    \left\|
    \begin{bmatrix}
        \Delta x \\
        \Delta u
    \end{bmatrix}
    \right\|_\infty \le \tau,\quad \Delta x = x-z,\quad \Delta u = u-v.
    \label{eq:tau_definition}
\end{equation}
\end{assumption}

Intuitively, the function $\sigma(z,v,\tau)$ captures how much the nonlinear dynamics can deviate from the linearized model within a neighborhood of radius $\tau$. 
\begin{remark}
\label{rem:linearization_error}
A natural choice for $\sigma(z,v, \tau)$ is the linearization error in its Lagrange form, as used in~\cite{leeman2025robust_TAC}.
Consider the Hessian of the $i^\text{th}$ component of $f$, i.e.,
\begin{equation}
H_i(\xi, \tilde w) \;=\;
\left.\frac{\partial^2 ( f(x,u) + E(x,u)w)_i}{\partial y^2}\right|_{y=\xi, w = \tilde w},
\end{equation}
with $y = (x,u)$, and where $\xi\in \mathbb{R}^{n_\mathrm{x}+n_\mathrm{u}}$ lies in the convex hull of the linearization point $(z,v)$ and the evaluation point $(x,u)$. We define the curvature bound $\mu_i\in\mathbb{R}^{n_\mathrm{x} \times n_\mathrm{x} }$
\begin{equation}
    \mu_i \defmath \tfrac{1}{2} 
    \max_{\substack{\xi \in \mathcal{C}, \tilde w \in \mathcal{B}^{\nx} \\ \|h\|_\infty \le 1}}
    |h^\top H_i(\xi, \tilde w) h|.
    \label{eq:def_mu}
\end{equation}
Then, \begin{equation}
    \sigma_i(z ,v, \tau ) =  \tau^2\mu_i + \|e_i^\top E(z,v)\|_1,
    \label{eq:rem_def}
\end{equation}
satisfies~\eqref{eq:LTV_uncertain}, where $e_i$ is the $i^\text{th}$ vector of the unit basis, and  $\sigma_i(z,v,\tau)$ is the $i^\text{th}$ component of $\sigma(z,v,\tau)$.
\end{remark}
Instead of the analytic over-approximation in \eqref{eq:rem_def}, $\sigma$ can also directly be tuned offline using data and sampling-based methods, analogous to domain randomization in \ac{RL}.

The reformulation \eqref{eq:LTV_uncertain}, \eqref{eq:tau_definition} is at the center of the presented \ac{RMPC} formulation, as it will enable the joint optimization of the nominal nonlinear trajectory~\eqref{eq:nom_nonlinear}, the disturbance-feedback controller based on the \ac{LTV} dynamics~\eqref{eq:LTV_aux} and the (time-varying) overbound $\bm \tau$.

We introduce an affine causal error feedback for the \ac{LTV} error dynamics \eqref{eq:LTV_aux} of the form
\begin{align}
    u_0 &= v_0 + \psi^\u_0 \label{eq:affine_contr}\\
    u_k &= v_k + \psi^\u_k + \sum_{j=0}^{k-1} K_{k,j} (\Delta x_j- \psi^\x_j),~k = 1,\ldots,T-1,\nonumber
\end{align}
where $K_{k,j} \in \R^{\nu \times \nx}$ are online optimized disturbance-feedback gains, and the auxiliary variables ${\psi}^\x_k \in \R^\nx$, ${\psi}^\u_k \in \R^\nu$ evolve according to the \ac{LTV} dynamics
\begin{align}
\psi^\x_{0} &= \Delta \x_0 = x(t) - z_0    \label{eq:nom_dyn}
\\
\psi^\x_{k+1} &= A(z_k, v_k) \psi^\x_{k} + B(z_k, v_k) \psi^\u_{k},~k = 0,\ldots,T-1.\nonumber
\end{align}
The following proposition, adapted from~\cite{leeman2023robust_CDC}, presents a parameterization of the affine disturbance feedback and robust reachable set for the nonlinear dynamics \eqref{eq:nonlinear_dyn}.
\begin{proposition}
    \label{prop:slp}
Consider a nominal nonlinear trajectory $(\Z,\V)$ satisfying~\eqref{eq:nom_nonlinear}, and auxiliary variables ${\psi}^\x_k \in \R^\nx$, ${\psi}^\u_k \in \R^\nu$ defined according to \eqref{eq:nom_dyn}.  The matrices  $\Px_{k,j}\in \R^{\nx\times \nw}$, $\Pu_{k,j}\in \R^{\nu\times \nw}$ satisfy
\begin{equation}
\label{eq:tube_propagation}
\begin{aligned}
    \Px_{j+1,j} &= \mathrm{diag}\left( \sigma(z_j,v_j,\tau_j)\right)\\
\Px_{k+1,j} &= A(z_k,v_k) \Px_{k,j} + B(z_k,v_k)\Pu_{k,j},\\
\end{aligned}
\end{equation}
$j=0,\ldots,T-1$, $k=j+1,\ldots, T-1$,
with the overbound
\begin{equation}
    \|[\P_{(k)},\psi_k]\|_\infty\le \tau_k,~k=0,\ldots,T-1,
\end{equation}
and $\psi_k \defmath (\psi^\x_k, \psi^\u_k)$ and $\Phi_{(k)} \defmath (\Phi_{(k)}^\x, \Phi_{(k)}^\u)$.
Then, the system \eqref{eq:nonlinear_dyn} with disturbance feedback~\eqref{eq:affine_contr}--\eqref{eq:nom_dyn} satisfies
\begin{equation}
\begin{aligned}
     x(k) &\in \mathcal{R}_\x( z_k,\psi^\x_k,\P^\x_{(k)}) \defmath \{ z_k + \psi^\x_k\}\bigoplus_{j=0}^{k -1}  \P^{\x}_{k,j} \B^{\nw},\\
     u(k)&\in \mathcal{R}_\u( v_k,\psi^\u_k,\P^\u_{(k)})\defmath \{ v_k + \psi^\u_k\}\bigoplus_{j=0}^{k -1}  \P^{\u}_{k,j} \B^{\nw}.
\end{aligned}
\label{eq:reach}
\end{equation}
\end{proposition}

\begin{proof}
The proof follows by considering the reachable set of the \ac{LTV} error dynamics
\begin{equation}
\begin{aligned}
\Delta x_{k+1} &\in \{A_k \Delta x_k + B_k \Delta u_k\} \oplus \Sigma_k \mathcal{B}^\nw, \\
\Delta x_0 &= x(t) - z_0,
\end{aligned}
\label{eq:error_uncertain_LTV_prop}
\end{equation}
where $\Sigma_k = \mathrm{diag}\left( \sigma(z_k,v_k,\tau_k)\right) $ overapproximates the disturbance of the nonlinear dynamics arising in~\eqref{eq:LTV_uncertain}, see also~\cite[Prop.~3.1.--3.3]{leeman2023robust_CDC} for a similar proof.
\end{proof}

By introducing the controller parametrization \eqref{eq:affine_contr}–\eqref{eq:nom_dyn} and the (over-)approximation of the reachable sets of the nonlinear dynamics~\eqref{eq:reach}, we can ensure robust constraint satisfaction by requiring that the state and input reachable sets $(\mathcal{R}_\x,\mathcal{R}_\u)$ lie within the constraint set $\mathcal{C}$.
Next, we formulate the optimization problem solved at each time step in the MPC loop.

\subsection{Finite-horizon robust nonlinear optimal control}
\label{sec:finite_horizon_robust_nonlinear_optimal_control}
Using the LTV error system response from Prop \ref{prop:slp}, we formulate a tractable robust nonlinear optimal control problem that jointly optimizes the nominal trajectory, affine disturbance-feedback, and overbounding terms to ensure robust constraint satisfaction.

As in standard MPC formulations, additional terminal ingredients are required to ensure recursive feasibility (see Sec.~\ref{sec:robust_mpc_framework}).
Therefore, the finite-horizon problem also uses a terminal set
\begin{equation}
\begin{aligned}
    \mathcal{X}_\f& \defmath  \{x\in \R^{\nx}|~c_{i,\f}\T x + b_{i,\f} \le 0,~ i= 1, \ldots, {n_\f}\},\label{eq:terminal_constraints}
    \end{aligned}
\end{equation}
with the origin in its interior, $0\in \text{int}(\mathcal{X}_\f)$, $c_{i,\f} \in \R^{\nx}$, and $b_{i,\f} \in \R$. 
Stability and performance will be ensured using a standard quadratic finite-horizon cost
\begin{equation}
\label{eq:JT_cost}
\begin{aligned}
&J_T(\Z,\V,\bm\psi,\bm\Phi)= \sum_{k=0}^{T-1} \ell(z_k+\psi_k^\x,\, v_k+\psi_k^\u) + \ell_\f(z_T+\psi_T^\x) \\
\end{aligned}
\end{equation}
with $\ell(x,u)$ from Eq.~\eqref{cost:setup} and $\ell_\f(x) = x^\top P x$ is a quadratic terminal cost.
The proposed \ac{RMPC} optimization problem is formulated as:
\begin{subequations}
\begin{align}
\min_{\substack{\bm \Phi, \Z, \V,\\ \tube, \bm\psi}}\quad & J_T( \Z,\V, \bm\psi, \bm \Phi), \label{eq:sls_cost}\\
\text{s.t.}\quad 
& z_{k+1} = f(z_k, v_k), ~k =0, \ldots, T-1 \label{eq:nonlinear_sls_nom_nl}\\
&\psi_0^\x= x(t) - z_0,\label{eq:nonlinear_sls_nom_ic}\\
& z_T = 0,\label{eq:nonlinear_sls_terminal_conditions}\\
& \Px_{j+1,j} = \text{diag}(\sigma(z_j, v_j, \tau_j)),\label{eq:nonlinear_sls_Px} \\
&\Px_{k+1,j} = A(z_k,v_k) \Px_{k,j} + B(z_k,v_k)\Pu_{k,j}, \nonumber\\
& \quad j = 0, \ldots, T-1, ~k = j+1, \ldots, T-1, \label{eq:sls_slp}\\
&\psi_{k+1}^\x = A(z_k,v_k) \psi_k^\x + B(z_k,v_k) \psi_k^\u,\nonumber \\
&\quad k=0,\ldots, T-1,\label{eq:sls_LTV_IC}\\
& \sum_{j=0}^{k-1}\| c_i\T \P_{k,j}\|_1 + c_i\T \sv{z_k + \psi^\x_k}{v_k +\psi^\u_k} + b_i \le 0,\nonumber\\
& \quad k = 0,\ldots, T-1,~ i=1,\ldots,n_\textrm{c}\label{eq:cons_nonlinear_SLS}\\
& \sum_{j=0}^{T-1} \| c_{i,\f}\T \Px_{T,j}\|_1 + c_{i,\f}\T \psi_T^\x + b_{i,\f} \le 0,\nonumber\\
&\quad ~i=1,\ldots,n_{\f},\label{eq:terminal_set}\\
&\|[\P_{(k)},\psi_k]\|_\infty\le \tau_k,~k=0,\ldots,T-1.\label{eq:nonlinearity_bounder}
\end{align}
\label{eq:nonlinear_sls}
\end{subequations}

The formulation~\eqref{eq:nonlinear_sls} integrates all components of the robust prediction and control design:
\begin{itemize}
    \item The nominal trajectory is propagated according to the nonlinear model~\eqref{eq:nonlinear_sls_nom_nl}, subject to the initial and terminal conditions~\eqref{eq:nonlinear_sls_nom_ic}–\eqref{eq:nonlinear_sls_terminal_conditions}. In particular,  $x(t)\neq z_0$ in~\eqref{eq:nonlinear_sls_nom_ic} allows for discrepancies between the nominal trajectory the uncertain system, which is crucial to provide a feasible candidate solution. 
Together with the terminal constraint, this is key to establishing recursive feasibility, a feature not present in previous formulations~\cite{leeman2025robust_TAC,leister2024robust}.
\item 
The disturbance propagation is captured through the LTV error dynamics~\eqref{eq:nonlinear_sls_Px}–\eqref{eq:sls_LTV_IC}, where the terms $\Px_{k,j}$ and $\Pu_{k,j}$ describe how model mismatch in $j$ steps affect the state and input in $k\geq j$ steps.
\item Robust constraint satisfaction is enforced by the tightened constraints~\eqref{eq:cons_nonlinear_SLS}–\eqref{eq:terminal_set}, which guarantee that both state and input constraints \eqref{eq:constraints} hold for all disturbances and modeling errors. This robust inclusion is enforced using the support function of the uncertainty sets \eqref{eq:tube_propagation}.
\item Finally, the auxiliary bounds~\eqref{eq:nonlinearity_bounder} ensure that the nonlinear modeling errors remain contained within the radius $\tau_k$, which is itself a decision variable optimized together with the trajectory and disturbance-feedback gains.
\end{itemize}

As per Prop.\ref{prop:slp}, the disturbance-feedback gains resulting from solving \eqref{eq:nonlinear_sls} ensures that the uncertain nonlinear system \eqref{eq:nonlinear_dyn} remains within the constraints $\mathcal{C}$ over the prediction horizon.

\section{Robust MPC Theoretical Analysis}
\label{sec:robust_mpc_framework}

This section shows the main theoretical properties of the proposed \ac{RMPC}: recursive feasibility, robust constraint satisfaction, and \ac{ISS} stability. 

The optimization problem~\eqref{eq:nonlinear_sls} is solved in receding horizon at each time step $t \in\mathbb{N}$ and we apply only the first optimal input 
\begin{equation}
    u(t) = v_0^\star + \psi^{\u\star}_0,
    \label{eq:optimal_input}
\end{equation}
where $v_0^\star, \psi^{\u\star}_0$ are the optimal solution for the current state $x(t)$. 
We then prove recursive feasibility of~\eqref{eq:nonlinear_sls}, i.e., if the problem is feasible at time $t$, it remains feasible at $t{+}1$ when applying the input~\eqref{eq:optimal_input} to the uncertain system~\eqref{eq:nonlinear_dyn}.

\subsection{Recursive feasibility}
\label{sec:recursive_feasibility}
First, to establish recursive feasibility, the terminal set $\mathcal{X}_\f$ needs to be chosen appropriately. 

A terminal controller $u = K_\f x$ is used, such that the linear system $A_\textrm{cl} \defmath A_\f + B_\f K_\f$ is stable, with $    A_\mathrm{f}\defmath A(0, 0),~ B_\mathrm{f}\defmath B(0,0)$.
The following constants bound the model errors in the terminal set:
\begin{equation*}
    \tau_\f \defmath \max_{\Delta x\in \mathcal{X}_\f} \|(\Delta x,  K_\f\Delta x)\|_\infty,\quad {\Sigma}_\mathrm{f} \defmath \text{diag}(\sigma(0,0,\tau_\f)).
\end{equation*}
\begin{assumption}
\label{assum:terminal_set}
   The polytopic terminal set $\mathcal{X}_\mathrm{f}$ is \ac{RPI}, i.e.,
    \begin{equation}
            A_\textrm{cl}\Xf \oplus \Sigma_\f \B^{\nx}\subseteq \mathcal{X}_\f, \label{eq:terminal_set_condition}
    \end{equation}
Furthermore, state and input constraints are satisfied in the terminal set, i.e., $\mathcal{X}_\f \times K_\f \mathcal{X}_\f \subseteq \mathcal{C}$.
\end{assumption}
The terminal set $\mathcal{X}_\f$ can be computed as the maximum \ac{RPI} $\mathcal{X}_\f$ set for the closed-loop dynamics $x_+ = A_\text{cl} x + w$ where $w\in \Sigma_\mathrm{f}\B^\nx $, see \cite{blanchini1999set} for an overview on computing invariant sets satisfying Assumption \ref{assum:terminal_set}.

\begin{theorem}
Given Assumptions~\ref{assum:sigma}, \ref{assum:terminal_set}, and suppose the optimization problem~\eqref{eq:nonlinear_sls}
feasible at $t=0$. Then, Problem~\eqref{eq:nonlinear_sls} remains feasible for all $t \in \mathbb{N}_0$, and the closed-loop system~\eqref{eq:nonlinear_dyn} with input~\eqref{eq:optimal_input} satisfies the constraints \eqref{eq:constraints} for all $t\in \mathbb{N}_0$ and all disturbances $w(t)\in \B^\nw$.
\label{thm:rec_feas}
\end{theorem}
\begin{proof}
Since the Jacobians in \eqref{eq:sls_slp} vary depending on the nonlinear nominal trajectory \eqref{eq:nonlinear_sls_nom_nl}, linear results \cite{sieber2025computationally,GOULART2006523} do not apply. 
Here, this issue is addressed by introducing the auxiliary trajectory $(\psi^\x,\psi^\u)$, which account for the mismatch between the nonlinear dynamics and their LTV approximation.
The detailed proof is provided in Appendix~\ref{sec:proof_rec_feas}.
\end{proof}

\subsection{Stability}
\label{sec:stab}
In the following, we analyze the stability of the closed-loop system~\eqref{eq:nonlinear_dyn} with \eqref{eq:optimal_input} under a cost function~$J_T$. We show that, with an appropriate choice of stage and terminal costs, the formulation ensures \ac{ISS} stability \cite{limon2009input}.

\begin{assumption}
\label{assum:lyap_decrease_terminal}
The terminal cost $\ell_\f(x) = x^\top P x$ in \eqref{eq:JT_cost} satisfies
\begin{equation}
A_{\mathrm{cl}}^\top P A_{\mathrm{cl}} - P + Q + K_\f^\top R K_\f \;\preceq\; 0.
\label{eq:ISS_lyap_dec}
\end{equation}
\end{assumption}

To analyze stability, we consider the optimal cost of Problem~\eqref{eq:nonlinear_sls} at time $t$
\begin{equation}
    V_t = \sum_{k=0}^{T-1} \ell(z_k^\star+\psi_k^{\x\star}, v_k^\star + \psi_k^{\u\star}) + \ell_\f(\psi^{\x\star}_T).
\end{equation}
\begin{theorem}
\label{thm:iss}
Suppose the conditions in Theorem~\ref{thm:rec_feas} and Assumption~\ref{assum:lyap_decrease_terminal} hold. 
Then $V_t$ is an ISS Lyapunov function and the closed-loop system is \ac{ISS} stable.
\end{theorem}

\begin{proof}
Considering the candidate solution from the proof of Theorem~\ref{thm:rec_feas}, the terminal cost condition (Asm.~\ref{assum:lyap_decrease_terminal}), and Lipschitz continuity of the cost $\ell,\ell_\f$ and dynamics $f$ on the compact set $\mathcal{C}$ ensures 
\begin{align*}
V_{t+1}-V_t \le -x(t)^\top Q x(t) + c_\ell\,\|w(t)\|,
\end{align*}
for some Lipschitz constant $c_\ell$. Moreover, the value function satisfies the following lower and upper bounds:
\begin{align*}
&V_t \geq \ell(z_0^\star+\psi_0^{\x,\star},\, v_0^\star+\psi_0^{\u,\star})
   \stackrel{\eqref{cost:setup},\eqref{eq:nonlinear_sls_nom_ic}}{\geq} x(t)^\top Q x(t), \\
&V_t \leq x(t)^\top P x(t), \quad x(t)\in\mathcal{X}_\f,~0\in\mathrm{int}(\mathcal{X}_\f).
\end{align*}
The local upper bound also implies an upper bound on the feasible set~\cite[Prop. 2.16]{rawlings2017model}. 
Thus, $V_t$ satisfies the three standard conditions of an ISS Lyapunov function and $V_t$ is an ISS Lyapunov function, and therefore the closed-loop system is ISS stable~\cite{limon2009input}.
\end{proof}

\section{Implementation via \ac{SCP}}
To solve the nonconvex optimization problem~\eqref{eq:nonlinear_sls}, we propose an efficient implementation that combines the flexibility of sequential convex programming (\ac{SCP})\cite{malyuta2022convex} with recent advances in numerical optimization for \ac{MPC} with disturbance-feedback\cite{leeman2024fast_NMPC}. At each \ac{SCP} iteration, a (convex) \ac{SOCP} is constructed by linearizing the nonlinear dynamics and convexifying the constraint tightening around the current nominal trajectory $(\Z,\V)$. The resulting \ac{SOCP} can be efficiently solved using the structure-exploiting solver of~\cite{leeman2024fast_NMPC}, and its solution is then used to update the trajectory. The cost further includes a quadratic regularization term on $\P$ to improve numerical stability of the optimization~\cite{leeman2024fast_NMPC}. The overall procedure is summarized in Algorithm~1.

Crucially, each \ac{SOCP} retains the numerical structure exploited by the fast-SLS solver~\cite{leeman2024fast_NMPC}, which is specifically designed for efficient \ac{MPC} with jointly optimized disturbance feedbacks. The solver alternates between Riccati recursion and \ac{QP} solve for the nominal dynamics.

\begin{tcolorbox}[title=Algorithm 1: Proposed Robust \ac{MPC}, colback=gray!5, colframe=black!30, before skip=6pt, after skip=6pt]
	\textbf{Inputs:}  horizon $T$, cost weights, constraints $\mathcal{C}$, $\mathcal{X}_\f$.\\
 \textbf{For $t = 0,1,2,\ldots$ (receding horizon):}
    \begin{enumerate}
        \item Measure current state $x(t)$.
    \item Initialize $(\Z,\V)$: shifted previous solution.
    \item Repeat until convergence \textbf{(SCP loop)}: 
    \begin{enumerate}
        \item Linearize~\eqref{eq:nonlinear_sls} around $(\Z,\V)$ to build an \ac{SOCP} approximation.
        \item Solve the \ac{SOCP} via~\cite{leeman2024fast_NMPC}, obtaining $(\Z^{\text{lin}\star},\V^{\text{lin}\star})$ and $\bm{\Phi}^\x$, $\bm{\Phi}^\u$, $\bm{\psi}^\x$, $\bm{\psi}^\u$.
        \item Update $(\Z,\V)\leftarrow (\Z,\V) + (\Z^{\text{lin}\star},\V^{\text{lin}\star})$. 
    \end{enumerate}
    \item Apply in closed loop: $u(t) = v^{\star}_0 + \psi^{\u\star}_0$.
\end{enumerate}
\end{tcolorbox}

\begin{remark}[Computational Efficiency]  
    To enhance computational efficiency, we adopt the \ac{RTI} scheme \cite{gros2020linear}: Instead of iterating the sequential convex program or the associated \ac{SOCP} solvers\cite{leeman2024fast_NMPC} to full convergence at every sampling step $t$, \ac{RTI} performs only one (or a few) iterations.
\end{remark}

\begin{remark}[Computational Complexity]  
The update of the nominal nonlinear trajectory, including Jacobian evaluation and the \ac{QP} solve, has a computational complexity of $\mathcal{O}(T(n_\mathrm{x}^3 + n_\mathrm{u}^3))$, matching that of state-of-the-art \ac{MPC} solvers such as~\cite{verschueren2022acados} and \ac{iLQR}~\cite{giftthaler2018family,neunert2016fast}. The Riccati-based disturbance-feedback update~\cite{leeman2024fast_NMPC} scales as $\mathcal{O}(T^2(n_\mathrm{x}^3 + n_\mathrm{u}^3))$ and constitutes the only additional computational cost compared to a nominal \ac{MPC}.
This represents a significant improvement over a naive implementation with complexity $\mathcal{O}(T^4 (n_\mathrm{x}^3 + n_\mathrm{u}^3)n_\mathrm{x}^3)$. 
\end{remark}

\section{Numerical Examples}

First, we consider a simple example of a cart-pole, where we illustrate the application of the proposed theory (Sec.~\ref{sec:inverted_pendulum}). 
Then, we study the more practical application to rocket landing, where we introduce approximations that enable rapid deployment without system-specific offline designs (Sec.~\ref{sec:rocket}). We evaluate our approach against soft-constrained MPC baselines\cite{chiu2022collision} (Sec. \ref{sec:comparison}).
Finally, we report computation times, where we additionally study consider a quadcopter~\cite{zhao2022tube} as well as on a space-grade CPU (Sec.~\ref{sec:computation_times}).
\ifthenelse{\boolean{anonymize}}%
{\footnote{All parameters are specified in the open-source implementation:\\https://anonymous.4open.science/r/robust-nonlinear-mpc/}}%
{\footnote{All parameters are specified in the open-source implementation:\\\url{https://github.com/antoineleeman/robust-nonlinear-mpc}}}
\subsection{Cart-pole System}
\label{sec:inverted_pendulum}
\begin{figure}[ht!]
    \centering
    \includegraphics[width=\linewidth]{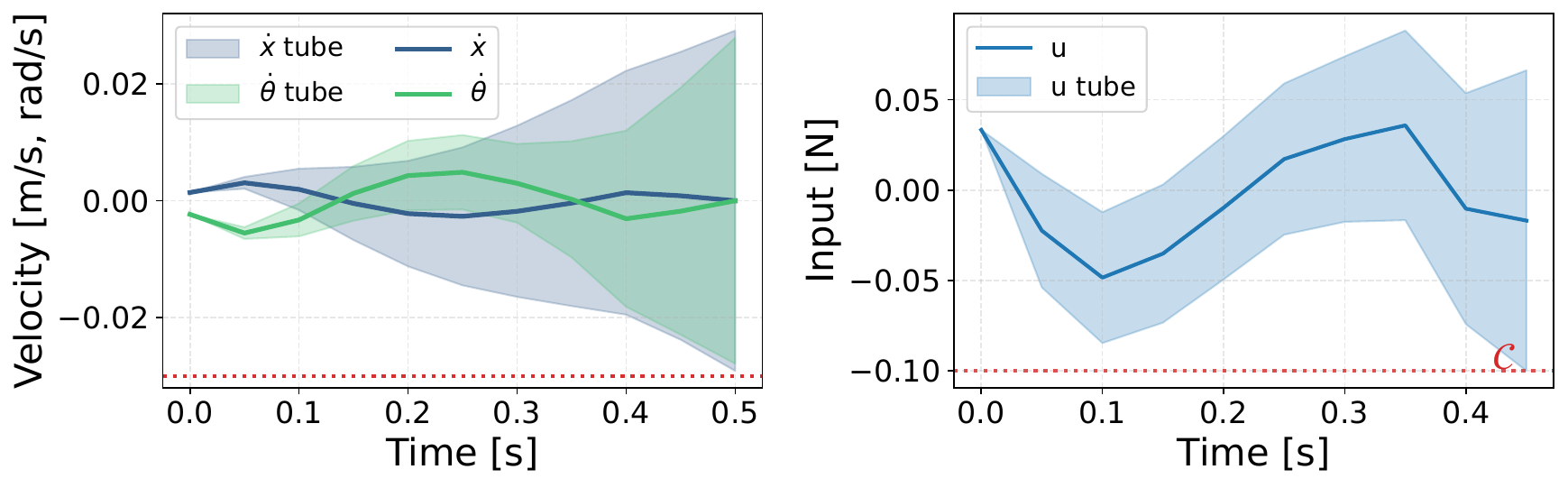}
\caption{Cart-pole implementation of the proposed \ac{RMPC} formulation \eqref{eq:nonlinear_sls}. The shaded regions show the disturbance reachable sets and the lines indicate predicted nominal states and input.}
    \label{fig:pendulum_plot}
\end{figure}
We first consider a cart-pole system, where we implement the proposed optimization problem \eqref{eq:nonlinear_sls}.
The constant $\mu$ from Remark \ref{rem:linearization_error} is obtained by sampling within the constraint set $\mathcal{C}$. The terminal set $\mathcal{X}_\f$ is designed via  \cite[Thm. 3.1]{kouvaritakis2016model}.

Fig.~\ref{fig:pendulum_plot} shows that the optimized reachable sets \eqref{eq:reach} with horizon $T=10$ lie within the constraints $\mathcal{C}$, demonstrating that the proposed robust MPC framework successfully enforces state and input constraints in the presence of disturbances. 
This example is implemented with a generic NLP solver (\texttt{sqpmethod} from CasADi~\cite{Andersson2019}). While it serves as a theoretical proof of concept, the solver gets extremely slow due to the large number of decision variables in the matrix~$\Phi$. Thus, for the following examples we implement the specialized solver (Alg.~1) that exploits the problem structure and scales efficiently to high-dimensional dynamics.

\subsection{Rocket landing with steerable thrust}
\label{sec:rocket}
\begin{figure*}[ht!]
\centering    \includegraphics[width=\linewidth]{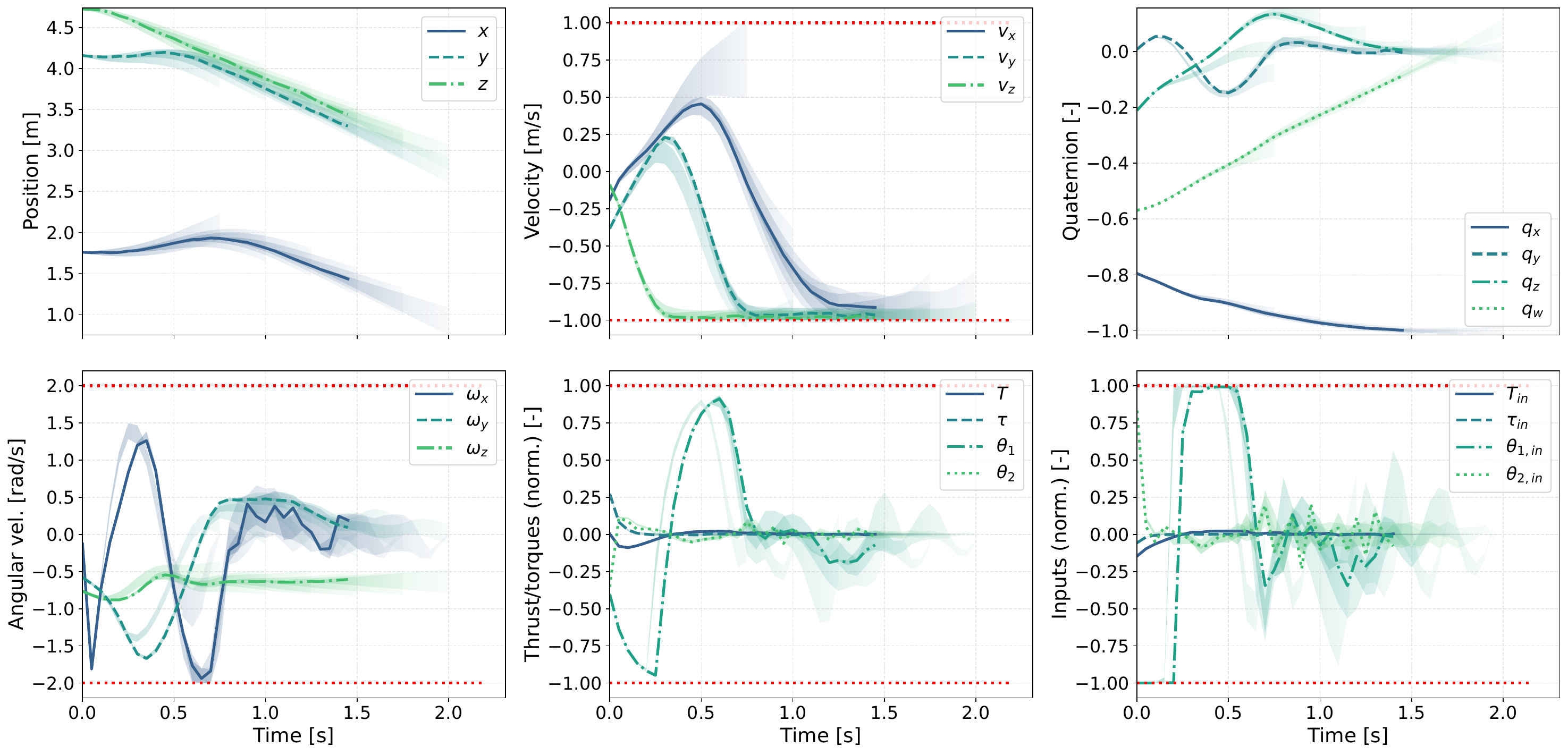}
    \caption{Closed-loop state and input trajectories for rocket landing with steerable thrust. The shaded regions show the disturbance reachable sets $\mathcal{R}^{\mathrm{x}},\mathcal{R}^{\mathrm{u}}$~\eqref{eq:reach} at selected time steps predicted over the horizon. The red dotted lines indicate state and input constraints.}
    \label{fig:Thrust_vectoring_landing}
\end{figure*}
To enable rapid deployment, we adopt a simplified implementation by setting $\mu =0$ (disregarding Hessian in \eqref{eq:rem_def}), omitting the terminal constraints~\eqref{eq:nonlinear_sls_terminal_conditions},\eqref{eq:terminal_set}, and fixing $\bm \psi = 0$. 
This approximation significantly reduces implementation complexity and allows for direct application without system-specific offline design, while relaxing some of the theoretical guarantees.

We adopt the rocket model from \cite{spannagl2021design}, which captures the translational and rotational dynamics
and an actuator with four internal states, yielding 17 states and 4 inputs. The control objective is to steer the rocket from arbitrary initial conditions to a vertical landing while minimizing a quadratic control objective, and robustly satisfying all safety constraints, on each state and input.
We use a sampling time of $50~[\mathrm{ms}]$, a prediction horizon of {$T=15$}, and a single \ac{RTI} iteration for both the SCP loop (Alg. 1) and the underlying \ac{SOCP} solver~\cite{leeman2024fast_NMPC}.

Each RTI step only takes about $20~[\mathrm{ms}]$ (real-time feasible).
The proposed \ac{RMPC} scheme successfully steers the rocket to a vertical landing while satisfying all state and input constraints, despite the presence of uncertainties.

\subsection{Comparison}
\label{sec:comparison}
We compare the proposed robust \ac{MPC} scheme against a soft-constraint \ac{MPC} baseline in Fig.~\ref{fig:compare_vel_inputs}. Both use identical dynamics, constraints, and cost functions. The soft-constraint formulation relaxes the constraints using slack variables penalized in the objective, while neither optimizing the $(\Phi^\x,\Phi^\u)$ nor guaranteeing robust constraint satisfaction. 
This ensures feasibility of the optimization problem at every time step, but may cause constraint violations when strict feasibility is not possible.
As opposed to our proposed approach, this soft-constraint approach results in significant violations of safety-critical constraints and incurs large oscillations.
An \ac{iLQR} implementation or its variations\cite{giftthaler2018family} would yield similar violations of the state constraints.
These results highlight the advantages of the proposed method in guaranteeing safety relying on heuristic relaxations and tuning. Moreover, the closed-loop behavior demonstrates minimal conservatism, as the system operates close to the constraints while maintaining robust feasibility.
\begin{figure}[ht!]
    \centering
    \includegraphics[width=\linewidth]{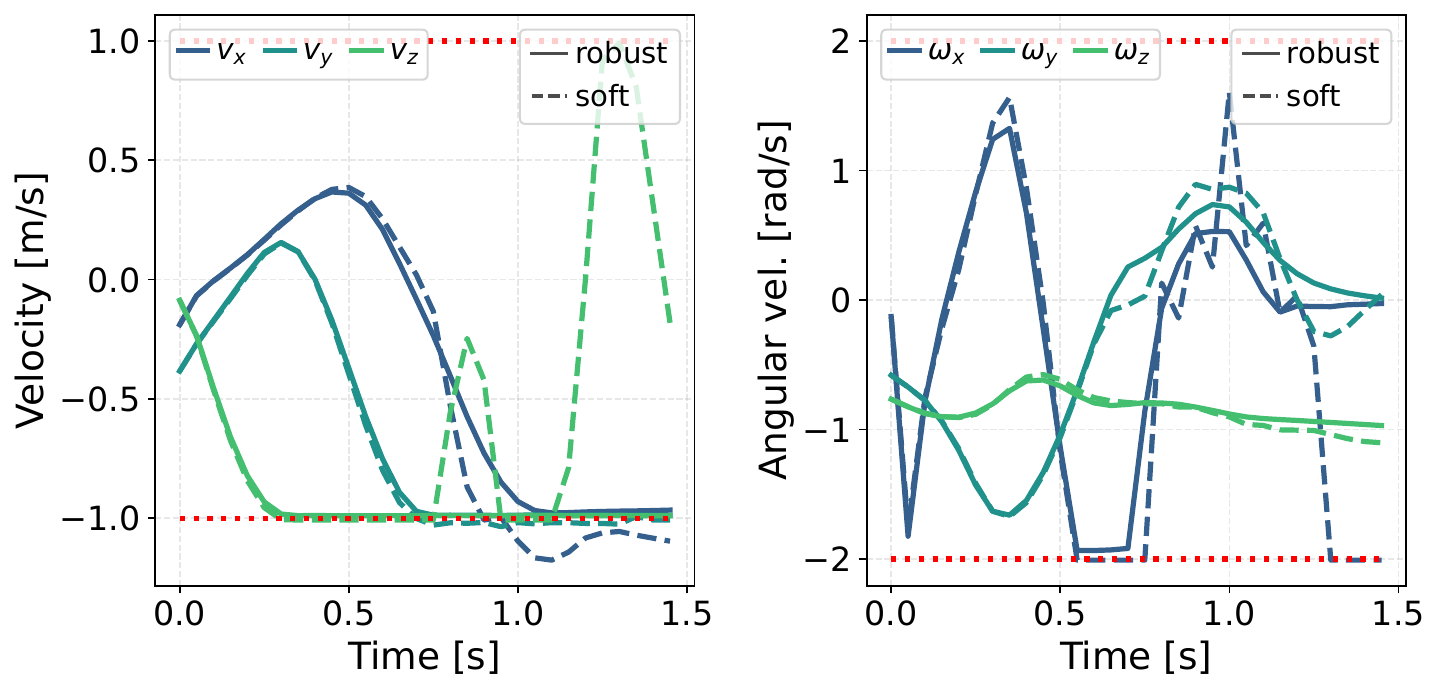}
\caption{Closed-loop velocity $v$ and angular velocity $\omega$ comparison between robust (ours) and soft-constraint \ac{MPC}\cite{chiu2022collision} for rocket landing with steerable thrust.}
    \label{fig:compare_vel_inputs}
\end{figure}

\subsection{Computation times}
\label{sec:computation_times}
Beyond the rocket, we implement a standard quadcopter~\cite{zhao2022tube} and a cart-pole example. 
We use the fast-SLS \ac{SOCP} solver~\cite{leeman2024fast_NMPC} 
which relies on a \ac{QP} solver that we generate with OSQP C-code~\cite{banjac2017embedded}.
We have implemented the backward Riccati recursions using Numba \ac{JIT} compiler \cite{lam2015numba} in parallel. Finally, we have evaluated the Jacobians of the dynamics using CasADi \cite{Andersson2019} with \ac{JIT}.The primary computational steps of Algorithm~1 are solving the \ac{QP}, performing the Riccati recursions (Ricc.), and evaluating the Jacobians (Jac.). Table~\ref{tab:times-vs-horizon} summarizes the per-iteration computation times across various horizon lengths and dynamics measured on a MacBook Pro (Apple M1, 8-core CPU, 16\,GB RAM, 2020). 
With a maximum runtime of $2.2~\textrm{[ms]}$ for the largest problem, Riccati recursions are computationally negligible compared to the dominant \ac{QP} step. The measured computation times are competitive with state-of-the-art (nominal) MPC solvers~\cite{9561438}, and are compatible with real-time deployment.

To demonstrate feasibility under highly limited computational resources, we also implemented the \ac{QP} solve with the LEON4 architecture provided by \ac{ESA}, a space-grade CPU optimized for fault tolerance and radiation hardness. 
Despite being significantly slower than desktop processors, 
these results indicate that the core optimization step can be executed onboard, showing that the overall algorithm is suitable for space \textit{guidance}, especially in combination with a low-level controller.

Notably, even though we solve a robust problem, the only additional cost over a standard nominal NMPC implementation using \ac{SCP}~\cite{verschueren2022acados} is the Riccati column updates required by the disturbance-feedback formulation, which represent only a small fraction of the total solve time in our examples. For comparison, solving the cart-pole problem with $N=15$ to convergence using CasADi IPOPT\cite{Andersson2019} required on average $5435.5\texttt{[ms]}$, i.e., more than $10^3 \times$ slower than our \ac{SCP} implementation (Alg.~1).




\begin{table}[ht!]
  \centering
  \caption{Average computation times [ms] vs. horizon length $N$.}
  \label{tab:times-vs-horizon}
  \sisetup{detect-weight,detect-inline-weight=math,table-number-alignment=right}
  \renewcommand{\arraystretch}{0.95}
  \begin{tabular}{
    l
    S[table-format=3.0]
    S[table-format=3.1]
    S[table-format=3.1]
    S[table-format=3.1]
  }
    \toprule
    {Dynamics $(n_x,n_u)$} & {$N$} & {Jac. [ms]} & {Ricc. [ms]} & {QP [ms]} \\
    \midrule
    \multirow{4}{*}{Cart-pole (4,1)}
      & 15 & {0.5} & {0.4} & {  0.2} \\
      & 30 & {1.0} & {0.5} & {  0.4} \\
      & 40 & {1.3} & {0.5} & {  0.6} \\
    \midrule
    \multirow{4}{*}{Quadcopter\cite{zhao2022tube} (12,4)}
      & 15 & {0.8} & {0.5} & {  2.8} \\
      & 30 & {2.2} & {1.0} & {  5.3} \\
      & 40 & {2.1} & {1.7} & {  4.8} \\
    \midrule
\multirow{4}{*}{Rocket landing \cite{spannagl2021design} (17,4)}
  & 15 & {1.2} & {1.0} & { 17.5} \\
  & 30 & {2.1} & {1.3} & { 56.4} \\
  & 40 & {2.7} & {2.2} & {128.9} \\
\cmidrule(lr){2-5}
   LEON4 (space-grade CPU): & 10 & n/a & n/a & 2787\, \\
    \bottomrule
  \end{tabular}
\end{table}

\section{Conclusion}
We introduced a scalable robust \ac{MPC} formulation for nonlinear systems that ensures robust constraint satisfaction, recursive feasibility, and input-to-state stability. The formulation jointly optimizes disturbance-feedback gains, a nominal trajectory, and model-error overbounds.
By integrating a disturbance-feedback MPC solver within a sequential convex programming framework, we achieve efficient joint optimization.
Overall, by combining theoretical guarantees with an efficient solver implementation, we provide a scalable framework for robust-by-design \ac{MPC} that requires no problem-specific tailoring and enables real-time deployment on agile robotic platforms. Its robust performance is further validated on a challenging rocket-landing problem with steerable thrust.

\ifthenelse{\boolean{anonymize}}{
}{
\section*{Acknowledgment}
The authors would like to thank Valentin Preda from \ac{ESA} for his valuable help with the LEON4 implementation.
}

\bibliographystyle{IEEEtran}
\bibliography{IEEEabrv,references}

\begin{thebibliography}{10}
\providecommand{\url}[1]{#1}
\csname url@rmstyle\endcsname
\providecommand{\newblock}{\relax}
\providecommand{\bibinfo}[2]{#2}
\providecommand\BIBentrySTDinterwordspacing{\spaceskip=0pt\relax}
\providecommand\BIBentryALTinterwordstretchfactor{4}
\providecommand\BIBentryALTinterwordspacing{\spaceskip=\fontdimen2\font plus
\BIBentryALTinterwordstretchfactor\fontdimen3\font minus
  \fontdimen4\font\relax}
\providecommand\BIBforeignlanguage[2]{{%
\expandafter\ifx\csname l@#1\endcsname\relax
\typeout{** WARNING: IEEEtran.bst: No hyphenation pattern has been}%
\typeout{** loaded for the language `#1'. Using the pattern for}%
\typeout{** the default language instead.}%
\else
\language=\csname l@#1\endcsname
\fi
#2}}

\bibitem{brunke2022safe}
L.~Brunke, M.~Greeff, A.~W. Hall, Z.~Yuan, S.~Zhou, J.~Panerati, and A.~P.
  Schoellig, ``Safe learning in robotics: From learning-based control to safe
  reinforcement learning,'' \emph{Annual Review of Control, Robotics, and
  Autonomous Systems}, vol.~5, no.~1, pp. 411--444, 2022.

\bibitem{jenelten2024dtc}
F.~Jenelten, J.~He, F.~Farshidian, and M.~Hutter, ``Dtc: Deep tracking
  control,'' \emph{Science Robotics}, vol.~9, no.~86, p. eadh5401, 2024.

\bibitem{malyuta2022convex}
D.~Malyuta, T.~P. Reynolds, M.~Szmuk, T.~Lew, R.~Bonalli, M.~Pavone, and
  B.~A{\c{c}}{\i}kme{\c{s}}e, ``Convex optimization for trajectory generation:
  A tutorial on generating dynamically feasible trajectories reliably and
  efficiently,'' \emph{IEEE Control Systems Magazine}, vol.~42, no.~5, pp.
  40--113, 2022.

\bibitem{rawlings2017model}
J.~B. Rawlings, D.~Q. Mayne, and M.~Diehl, \emph{Model predictive control:
  theory, computation, and design}.\hskip 1em plus 0.5em minus 0.4em\relax Nob
  Hill Publishing Madison, WI, 2017, vol.~2.

\bibitem{singh2023robust}
S.~Singh, B.~Landry, A.~Majumdar, J.-J. Slotine, and M.~Pavone, ``Robust
  feedback motion planning via contraction theory,'' \emph{The International
  Journal of Robotics Research}, vol.~42, no.~9, pp. 655--688, 2023.

\bibitem{zhao2022tube}
P.~Zhao, A.~Lakshmanan, K.~Ackerman, A.~Gahlawat, M.~Pavone, and N.~Hovakimyan,
  ``Tube-certified trajectory tracking for nonlinear systems with robust
  control contraction metrics,'' \emph{IEEE Robotics and Automation Letters},
  vol.~7, no.~2, pp. 5528--5535, 2022.

\bibitem{Messerer2021AnFeedback}
F.~Messerer and M.~Diehl, ``{An Efficient Algorithm for Tube-based Robust
  Nonlinear Optimal Control with Optimal Linear Feedback},'' in \emph{Proc.
  60th Conf. on Decision and Control}.\hskip 1em plus 0.5em minus 0.4em\relax
  IEEE, 2021, pp. 6714--6721.

\bibitem{Kim2022JointNonlinearities}
T.~Kim, P.~Elango, and B.~A{\c{c}}{\i}kme{\c{s}}e, ``Joint synthesis of
  trajectory and controlled invariant funnel for discrete-time systems with
  locally lipschitz nonlinearities,'' \emph{International Journal of Robust and
  Nonlinear Control}, vol.~34, no.~6, pp. 4157--4176, 2024.

\bibitem{GOULART2006523}
P.~J. Goulart, E.~C. Kerrigan, and J.~M. Maciejowski, ``Optimization over state
  feedback policies for robust control with constraints,'' \emph{Automatica},
  vol.~42, no.~4, pp. 523--533, 2006.

\bibitem{anderson2019system}
J.~Anderson, J.~C. Doyle, S.~H. Low, and N.~Matni, ``System level synthesis,''
  \emph{Annual Reviews in Control}, vol.~47, pp. 364--393, 2019.

\bibitem{leeman2025robust_TAC}
A.~P. Leeman, J.~Köhler, A.~Zanelli, S.~Bennani, and M.~N. Zeilinger, ``Robust
  nonlinear optimal control via system level synthesis,'' \emph{IEEE Trans. on
  Automatic Control}, vol.~70, no.~7, pp. 4780--4787, 2025.

\bibitem{leister2024robust}
D.~D. Leister and J.~P. Koeln, ``Robust model predictive control for nonlinear
  discrete-time systems using iterative time-varying constraint tightening,''
  in \emph{Proc. American Control Conference (ACC)}.\hskip 1em plus 0.5em minus
  0.4em\relax IEEE, 2025, pp. 71--78.

\bibitem{leeman2023robust_CDC}
A.~P. Leeman, J.~Sieber, S.~Bennani, and M.~N. Zeilinger, ``Robust optimal
  control for nonlinear systems with parametric uncertainties via system level
  synthesis,'' in \emph{Proc. 62nd IEEE Conference on Decision and Control
  (CDC)}.\hskip 1em plus 0.5em minus 0.4em\relax IEEE, 2023, pp. 4784--4791.

\bibitem{blanchini1999set}
F.~Blanchini, ``Set invariance in control,'' \emph{Automatica}, vol.~35,
  no.~11, pp. 1747--1767, 1999.

\bibitem{sieber2025computationally}
J.~Sieber, A.~Didier, and M.~N. Zeilinger, ``Computationally efficient system
  level tube-{MPC} for uncertain systems,'' \emph{Automatica}, vol. 180, p.
  112466, 2025.

\bibitem{limon2009input}
D.~Limon, T.~Alamo, D.~M. Raimondo, D.~M. De~La~Pe{\~n}a, J.~M. Bravo,
  A.~Ferramosca, and E.~F. Camacho, ``Input-to-state stability: a unifying
  framework for robust model predictive control,'' \emph{Nonlinear Model
  Predictive Control: Towards New Challenging Applications}, pp. 1--26, 2009.

\bibitem{leeman2024fast_NMPC}
A.~P. Leeman, J.~K{\"o}hler, F.~Messerer, A.~Lahr, M.~Diehl, and M.~N.
  Zeilinger, ``Fast system level synthesis: Robust model predictive control
  using {Riccati} recursions,'' \emph{IFAC-PapersOnLine}, vol.~58, no.~18, pp.
  173--180, 2024.

\bibitem{gros2020linear}
S.~Gros, M.~Zanon, R.~Quirynen, A.~Bemporad, and M.~Diehl, ``From linear to
  nonlinear {MPC}: bridging the gap via the real-time iteration,''
  \emph{International Journal of Control}, vol.~93, no.~1, pp. 62--80, 2020.

\bibitem{verschueren2022acados}
R.~Verschueren, G.~Frison, D.~Kouzoupis, J.~Frey, N.~v. Duijkeren, A.~Zanelli,
  B.~Novoselnik, T.~Albin, R.~Quirynen, and M.~Diehl, ``acados—a modular
  open-source framework for fast embedded optimal control,'' \emph{Mathematical
  Programming Computation}, vol.~14, no.~1, pp. 147--183, 2022.

\bibitem{giftthaler2018family}
M.~Giftthaler, M.~Neunert, M.~St{\"a}uble, J.~Buchli, and M.~Diehl, ``A family
  of iterative gauss-newton shooting methods for nonlinear optimal control,''
  in \emph{2018 IEEE/RSJ Int. Conf. on Intelligent Robots and Systems
  (IROS)}.\hskip 1em plus 0.5em minus 0.4em\relax IEEE, 2018, pp. 1--9.

\bibitem{neunert2016fast}
M.~Neunert, C.~De~Crousaz, F.~Furrer, M.~Kamel, F.~Farshidian, R.~Siegwart, and
  J.~Buchli, ``Fast nonlinear model predictive control for unified trajectory
  optimization and tracking,'' in \emph{Proc. IEEE Int. Conf. on robotics and
  automation (ICRA)}.\hskip 1em plus 0.5em minus 0.4em\relax IEEE, 2016, pp.
  1398--1404.

\bibitem{chiu2022collision}
J.-R. Chiu, J.-P. Sleiman, M.~Mittal, F.~Farshidian, and M.~Hutter, ``A
  collision-free mpc for whole-body dynamic locomotion and manipulation,'' in
  \emph{Proc. Int. Conf. on robotics and automation (ICRA)}.\hskip 1em plus
  0.5em minus 0.4em\relax IEEE, 2022, pp. 4686--4693.

\bibitem{kouvaritakis2016model}
B.~Kouvaritakis and M.~Cannon, ``Model predictive control,'' \emph{Switzerland:
  Springer International Publishing}, vol.~38, no. 13-56, p.~7, 2016.

\bibitem{Andersson2019}
J.~A.~E. Andersson, J.~Gillis, G.~Horn, J.~B. Rawlings, and M.~Diehl, ``{CasADi
  -A software framework for nonlinear optimization and optimal control},''
  \emph{Mathematical Programming Computation}, vol.~11, no.~1, pp. 1--36, 2019.

\bibitem{spannagl2021design}
L.~Spannagl, E.~Hampp, A.~Carron, J.~Sieber, C.~A. Pascucci, A.~U. Zgraggen,
  A.~Domahidi, and M.~N. Zeilinger, ``Design, optimal guidance and control of a
  low-cost re-usable electric model rocket,'' in \emph{2021 IEEE/RSJ Int. Conf.
  on Intelligent Robots and Systems (IROS)}.\hskip 1em plus 0.5em minus
  0.4em\relax IEEE, 2021, pp. 6344--6351.

\bibitem{banjac2017embedded}
G.~Banjac, B.~Stellato, N.~Moehle, P.~Goulart, A.~Bemporad, and S.~Boyd,
  ``Embedded code generation using the osqp solver,'' in \emph{Proc. IEEE 56th
  Annual Conference on Decision and Control (CDC)}.\hskip 1em plus 0.5em minus
  0.4em\relax IEEE, 2017, pp. 1906--1911.

\bibitem{lam2015numba}
S.~K. Lam, A.~Pitrou, and S.~Seibert, ``Numba: A llvm-based python jit
  compiler,'' in \emph{Proceedings of the Second Workshop on the LLVM Compiler
  Infrastructure in HPC}, 2015, pp. 1--6.

\bibitem{9561438}
B.~E. Jackson, T.~Punnoose, D.~Neamati, K.~Tracy, R.~Jitosho, and
  Z.~Manchester, ``Altro-c: A fast solver for conic model-predictive control,''
  in \emph{Proc. IEEE Int. Conf. on Robotics and Automation (ICRA)}, 2021, pp.
  7357--7364.

\end{thebibliography}
\newpage
\appendix

\subsection{Proof of Theorem~\ref{thm:rec_feas}}
\label{sec:proof_rec_feas}
The proof proceeds in three steps.
First, we introduce a feasible candidate sequence and verify that it satisfies the disturbance propagation~\eqref{eq:sls_slp} together with the nominal nonlinear dynamics~\eqref{eq:nonlinear_sls_nom_nl}.
Second, we construct a shifted version of $\psi$ ensuring that the initial condition~\eqref{eq:nonlinear_sls_nom_ic} holds, while preserving the terminal set condition~\eqref{eq:terminal_set}.
Finally, we show that the resulting sequence satisfies the state and input constraints~\eqref{eq:cons_nonlinear_SLS} as well as the linearization error bound~\eqref{eq:nonlinearity_bounder}.

\textbf{1)} We denote ${\P}^{\x\star}$, ${\P}^{\u\star}$, $\bm \sigma^\star$, ${\Z}^\star$, $\V^\star$ the optimal value of the optimization problem~\eqref{eq:nonlinear_sls} at time step $t$ with initial condition $x(t)$. Then, the following shifted trajectory 
    \begin{equation}
\begin{aligned}
   \bar \Z &= (\bar z_0,\ldots,\bar z_T) = (z_1^\star,\ldots,z_T^\star,0), \\
   \bar \V &= (\bar v_0,\ldots,\bar v_{T-1}) = (v_1^\star,\ldots,v_{T-1}^\star,0).
\end{aligned}
\label{eq:feas_nom}
\end{equation}
and block-shifted matrices
\begin{equation}
    \label{eq:feas_phi_xw}
\begin{aligned}
            &\bar{\P}^\x_{k, j} = {\P}^{\x\star}_{k+1, j+1},\\
            &\quad j = 0, \ldots, T-1,~ k = j+1, \ldots, T-1,\\
            &\bar{\P}^\x_{T, j} = A_\textrm{cl}{\P}^{\x\star}_{T, j +1},~ j = 0, \ldots, T-2,\\
            &\bar{\P}^\x_{T, T-1} = \Sigma_\f.
\end{aligned}
    \end{equation}

    \begin{equation}
        \begin{aligned}
            &\bar{\P}^\u_{k, j} = {\P}^{\u\star}_{k+1, j+1},\\
            &\quad  j = 0, \ldots, T-1,~ k = j+1, \ldots, T-1,\\
            &\bar{\P}^\u_{T-1, j} = K_\f {\P}^{\x\star}_{T-1, j +1},~ j = 0, \ldots, T-2,
            \end{aligned}
    \end{equation}
satisfy the constraints~\eqref{eq:sls_slp},~\eqref{eq:nonlinear_sls_nom_nl}, and~\eqref{eq:nonlinear_sls_terminal_conditions} with the shifted uncertainty bounds
\begin{equation}
    \bar\tau_k \coloneqq \tau_{k+1}^\star,\quad k=0,\ldots,T-1,\qquad 
    \bar\tau_T \coloneqq \tau_\f.
    \label{eq:feas_tau}
\end{equation}

Indeed, as the nominal trajectory $\bar \Z$ and $\bar \V$ is the shifted version of an optimal sequence, it still satisfies~\eqref{eq:nonlinear_sls_nom_nl}, and~\eqref{eq:nonlinear_sls_terminal_conditions}, considering also $f(0,0)=0$.
    Thus, the shifted $\bar \Phi$ also satisfy the equality
\begin{equation}
    \begin{aligned}
        \bar{\Phi}^\x_{j+1,j} &= \textrm{diag}(\sigma(\bar z_j,\bar v_j,\bar \tau_j)),\\
        \bar{\Phi}^\x_{k+1,j} &=  A(\bar z_k,\bar v_k) \bar{\Phi}^\x_{k,j} + B(\bar z_k,\bar v_k)\bar{\Phi}^\u_{k,j},\\
        \end{aligned}
\label{eq:slp_single_matrix}
\end{equation}
$j = 0, \ldots, T-1, k = j+1, \ldots, T-1$, i.e., the constraint~\eqref{eq:sls_slp} holds.   
This choice ensures that the block-shifted matrices $\bar\Sigma_j$ used in the definition of $\bar\P^\x$ remain consistent with the overbounding of the nonlinear dynamics. Hence, $\|[\bar{\psi}_k,\bar{\P}_{(k)}]\|_\infty \le \bar{\tau}_k$ for $k=0,\ldots,T$, i.e., \eqref{eq:nonlinearity_bounder} is preserved under the shift.

\textbf{2)}
We denote $\bm{\psi}^{\x\star}$, $\bm{\psi}^{\u\star}$, the optimal values of the optimization problem~\eqref{eq:nonlinear_sls} at time step $t$ with initial condition $x(t)$. 
A naive shift of the trajectory $(\bm{\psi}^\x,\bm{\psi}^\u)$ is not feasible for~\eqref{eq:nonlinear_sls_nom_ic}. Hence, we define $\bar w\in \B^\nx$ as an equivalent disturbance such that
\begin{equation}
    {\P}^{\x\star}_{1,0} \bar w \defmath  x(t+1) - z_1^\star -  \psi^{\x\star}_{1}.
\label{eq:w_bar}
\end{equation}

By construction, the difference $x(t+1) - \bigl(z_1^\star + \psi^{\x\star}_{1}\bigr)$ corresponds exactly to the deviation between the nonlinear system and its linearized prediction.  
Since this deviation is bounded componentwise by $\sigma(z_0^\star,v_0^\star,\tau_0)$ (cf. \eqref{eq:nonlinear_sls_Px}), there always exists a disturbance realization $\bar w \in \B^\nx$ such that \eqref{eq:w_bar} holds.

The following sequence is proposed as a feasible candidate:
\begin{equation}
\begin{aligned}
    \bar \psi_0^\x   &= x(t+1) - z_1^\star,\\
    \bar \psi_k^\u   &= {\psi}_{k+1}^{\u\star} + {\P}_{k,0}^{\u\star} \bar w, && k = 1, \ldots, T-1,\\
    \bar \psi_k^\x   &= {\psi}_{k+1}^{\x\star} + {\P}_{k,0}^{\x\star} \bar w, && k = 1, \ldots, T-1,\\
    \bar \psi^\x_T   &= A_\textrm{cl} \bar \psi^\x_{T-1},\\
    \bar \psi^\u_{T-1} &= K_\f \bar \psi^\x_{T-1},
\end{aligned}
\label{eq:feas_psi}
\end{equation}
where $\bar w$ is given by~\eqref{eq:w_bar} and $x(t+1)$ by~\eqref{eq:nonlinear_dyn} with optimal input \eqref{eq:optimal_input}.  
Multiplying \eqref{eq:slp_single_matrix} for $j=0$ by $\bar w$ yields
\begin{equation}
\begin{aligned}
    \P_{k+1,0}^{\x\star} \bar w = A(z_k^\star,v_k^\star) \P_{k,0}^{\x\star} \bar w + B(z_k^\star,v_k^\star) \P_{k,0}^{\u\star} \bar w,
\end{aligned}
\end{equation}
which, by the principle of superposition, shows that $\bar\psi^\x, \bar\psi^\u$ satisfy\eqref{eq:sls_LTV_IC}.

\textbf{3)}
Lastly, we write the reachable sets of the proposed feasible candidate solution as

    \begin{equation}
\bar{\mathcal{R}}_{ k}^\x = \{\bar z_k + \bar{\psi}^\x_k \} \bigoplus_{j=0}^{k-1}  \bar{\P}^{\x}_{k,j} \B^{\nx}.
\end{equation}%
Likewise, we denote
\begin{equation}
{\mathcal{R}}_{ k+1}^{\x\star } = \{ z_{k+1}^\star+ {\psi}^{\x\star}_{k+1} \} \bigoplus_{j=0}^{k}  \P^{\x\star}_{k+1,j} \B^{\nx}.
\end{equation}

The tube inclusion of the reachable set, as required to show recursive feasibility, writes
    \begin{equation}
        \bar{\mathcal{R}}_{ k}^\x  \subseteq \mathcal{R}^{\star\x }_{ k+1}, ~k = 0, \ldots, T-1,
\label{eq:set_inclu_x}
    \end{equation}
    \begin{equation}
        \bar{\mathcal{R}}_{ k}^\u  \subseteq \mathcal{R}^{\star\u }_{ k+1}, ~k = 0, \ldots, T-2,
        \label{eq:set_inclu_u}
    \end{equation}
Since the nonlinear trajectory is shifted as per~\eqref{eq:feas_nom} with
$\bar z_k = z_{k+1}^\star$, and $\bar v_k = v_{k+1}^\star$,
the inclusions~\eqref{eq:set_inclu_x} simplify to
\begin{equation}
    \{\bar{\psi}^\x_k \} \bigoplus_{j=0}^{k-1}  \bar{\P}^{\x}_{k,j} \B^{\nx}\subseteq \{ {\psi}^{\x\star}_{k+1} \} \bigoplus_{j=0}^{k}  \P^{\x\star}_{k+1,j} \B^{\nx}.
    \label{eq:set_inclusion_proof}
\end{equation}
As per the construction of $\bar{\P}_{\x}$ in~\eqref{eq:feas_phi_xw}, each block of the feasible system response corresponds, by design, to a laterally shifted block of the optimal solution
\begin{equation}
    \bar{\P}^{\x}_{k,j} = \P^{\x\star}_{k+1,j+1}
    \label{eq:mat_feas_phi_x_shifted}
\end{equation}
for $k = 1, \ldots, T-1$ and $j = 0, \ldots, k-1$. 

Hence, the set inclusion~\eqref{eq:set_inclusion_proof} is equivalent to
\begin{equation}
        \{ \bar{\psi}^\x_k \} \in \{{\psi}^{\x\star}_{k+1} \} \oplus  {\P}_{(k+1,0)}^{\x\star} \B^{\nx},
\label{eq:proof:lemma43}
\end{equation}
which is satisfied as per construction of $\bar \psi^\x$ in ~\eqref{eq:feas_psi} since $\bar w \in \B^{\nx}$ as per Eq. \eqref{eq:w_bar}. Satisfaction of the input constraints follows analogously.
   Hence, it is guaranteed that the reachable sets~\eqref{eq:reach} satisfy~\eqref{eq:set_inclu_x}--\eqref{eq:set_inclu_u} and thus the proposed candidate sequence satisfies the constraints~\eqref{eq:cons_nonlinear_SLS}. 
As for the terminal constraints \eqref{eq:terminal_set}, we need to show that $ \bar{\psi}^\x_T \oplus \bar{\P}^{\x}_{(T)} \B^{\nx T} \subseteq \mathcal{X}_\f$.
Using the definition of $\bar \psi^\x_T$ from \eqref{eq:feas_psi} and the definition of $\bar{\P}^{\x}_{(T)}$ from \eqref{eq:feas_phi_xw} with $\bar{w}\in\B^{\nx}$, we have
\begin{align*}
&\bar{\psi}^\x_T \oplus \bar{\P}^{\x}_{(T)} \B^{\nx T} \\
\subseteq &A_\textrm{cl}\left[{\psi}^{\x\star}_T\bigoplus_{j=0}^{T-1}\P^{\x\star}_{T,j}\B^{\nx}\right]\oplus \bar{\P}^\x_{T, T-1} \B^{\nx}\\
\subseteq& A_\textrm{cl}\mathcal{X}_\f \oplus \Sigma_\f\B^{\nx}
\subseteq \mathcal{X}_\f,
\end{align*}
where the penultimate step used the fact that ${\psi}^{\x\star}_T \oplus {\P}^{\x\star }_{(T)} \B^{\nx T} \subseteq \mathcal{X}_\f$ and the last step follows from Assumption~\ref{assum:terminal_set}.
Finally, we need to show that the constraint \eqref{eq:nonlinearity_bounder}, which bounds the linearization error via $\tau$, is satisfied by the proposed candidate solution, i.e., that the following inclusion holds $k = 0, \ldots, T-1$
\begin{equation}
\begin{aligned}
    \|[\bar{\psi}_k, \bar{\P}_{(k)} ]\|_\infty\le   \|[ \psi_{k+1}^{\star}, {\P}_{(k+1)}^{\star } ]\|_\infty \le \tau_{k+1}^\star = \bar \tau_k.
    \label{eq:proof_inclu_tau}
\end{aligned}
\end{equation}

We can use~\eqref{eq:mat_feas_phi_x_shifted} and the definition of the infinity norm, to simplify the expression~\eqref{eq:proof_inclu_tau} to
\begin{equation}
\label{eq:proof_inclu_psi_phi}
        \|\bar \psi_k \|_\infty\le   \|[ \psi_{k+1}^{\star}, {\P}^{\star}_{k+1,0} ]\|_\infty.
\end{equation}
The inequality is satisfied as per~\eqref{eq:proof:lemma43}.

Thus, recursive feasibility follows from the shifted candidate solution, which remains feasible at the next step.  
Constraint satisfaction is preserved since the tightened inequalities hold for all admissible disturbances.  

\hfill$\blacksquare$
\end{document}